\documentclass[12pt,twoside]{article}

\usepackage{fancyhdr}
\usepackage{amsmath}
\usepackage{amscd}
\usepackage{amssymb}
\usepackage{amsfonts}
\usepackage{euscript}
\usepackage{oldgerm}
\usepackage{calligra}
\usepackage{mathrsfs}
\usepackage{hyperref}
\usepackage{url}
\usepackage{lastpage}
\usepackage{multirow}
\usepackage{graphicx}
\usepackage{stmaryrd}
\usepackage{wasysym}
\usepackage{pifont}

\usepackage{draftwatermark}
\SetWatermarkAngle{45}
\SetWatermarkLightness{0.9}
\SetWatermarkFontSize{5cm}
\SetWatermarkScale{2}
\SetWatermarkText{ }

\renewcommand{\thefootnote}{\fnsymbol{footnote}}

\newcommand{\pa}{\textsf{\textsl{PA}}}
\newcommand{\q}{\textsf{\textsl{Q}}}

\long\def\sfootnote[#1]#2{\begingroup
\def\thefootnote{\fnsymbol{footnote}}
\footnote[#1]{#2}\endgroup}

\newtheorem{theorem}{Theorem}[section]

\newtheorem{lemma}[theorem]{Lemma}
\newtheorem{corollary}[theorem]{Corollary}
\newtheorem{proposition}[theorem]{Proposition}

\newtheorem{remark}[theorem]{Remark}

\newenvironment{proof}{\noindent\mbox{\bf Proof.}}
{\hfill\mbox{\ding{111}}\bigskip}

\begin{document}

\pagestyle{fancy}
\lhead[{\tt page \thepage \ (of \pageref{LastPage})}]{\qquad\qquad  {\bf  G\"odel--Rosser's Incompleteness Theorems for non-{\sc re} Theories}}
\chead[]{}
\rhead[{\bf  G\"odel--Rosser's Incompleteness Theorems for non-{\sc re} Theories}\qquad\qquad ]{{\tt page \thepage \ (of \pageref{LastPage})}}
\lfoot[\copyright\ {\sc  Saeed Salehi \& Payam Seraji 2016}]{$\mathcal{S}\alpha\epsilon\epsilon\partial\mathcal{S}
\alpha\ell\epsilon\hbar\imath\!\centerdot\!\texttt{ir}$}
\cfoot[{\footnotesize {\tt  }}]{{\footnotesize {\tt  }}}
\rfoot[$\mathcal{S}\alpha\epsilon\epsilon\partial\mathcal{S}
\alpha\ell\epsilon\hbar\imath\!\centerdot\!\texttt{ir}$]{\copyright\ {\sc Saeed Salehi \& Payam Seraji 2016}}
\renewcommand{\headrulewidth}{1pt}
\renewcommand{\footrulewidth}{1pt}
\thispagestyle{empty}

\begin{table}
\begin{center}
\hspace{0.75em}
\begin{tabular}{|| c || l  | l ||}

\hline
 \multirow{7}{*}{\includegraphics[scale=0.5]
 {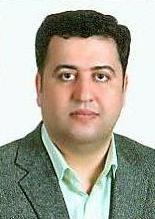}}&    &
 \multirow{7}{*}{ \ \ } \ \ \ \ \\
 &     \ \ {\large{\sc Saeed Salehi}}  \ \  \ & \ \    Tel: \, +98 (0)41 3339 2905      \\
 &   \ \ Department of Mathematics \ \  \ & \ \ Fax: \ +98 (0)41 3334 2102    \\
 &   \ \ University of Tabriz \ \ \  & \ \ E-mail: \!\!{\tt /root}{\sf @}{\tt SaeedSalehi.ir/}     \\
 &  \ \ P.O.Box 51666--17766 \ \ \ &   \ \ \ \ {\tt /SalehiPour}{\sf @}{\tt TabrizU.ac.ir/}     \\
 &   \ \ Tabriz, IRAN \ \ \ & \ \ Web: \  \ {\tt http:\!/\!/SaeedSalehi.ir/}    \\
 &    &    \\
 \hline
\hline
 \multirow{7}{*}{\includegraphics[scale=0.7]
 {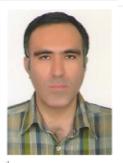}}&    &
 \multirow{7}{*}{ \ \ } \ \ \ \ \\
 &     \ \ {\large{\sc Payam Seraji}}  \ \  \ &  {\tt  http://mathoverflow.net/}           \\
 &   \ \ Department of  Mathematics \ \  \ & {\tt users/65878/payam-seraji}      \\
 &   \ \ University of Tabriz  \ \ \  & \ \ E-mail: \!\!{\tt   p$_-$seraji}{\sf @}{\tt  Yahoo.com}     \\
 &  \ \  P.O.Box 51666--17766 \ \ \ &   \ \ \ \ {\tt p$_-$seraji}{\sf @}{\tt TabrizU.ac.ir}      \\
 &    \ \ Tabriz, IRAN \ \ \ &    \\
 &    &    \\
 \hline
\end{tabular}
\end{center}
\bigskip
\end{table}

\vspace{1.5em}

\begin{center}

\bigskip

{\bf {\Large  G\"odel--Rosser's Incompleteness Theorems  \\[0.5em] for Non--Recursively Enumerable Theories
}}
\end{center}

\vspace{1.5em}

\begin{abstract}\noindent
G\"odel's First Incompleteness Theorem is generalized to definable theories, which are not necessarily recursively enumerable, by using a couple of syntactic-semantic notions;  one   is the consistency of a theory with the set of all true $\Pi_n$-sentences or equivalently the $\Sigma_n$-soundness of the theory, and the other is $n$-consistency the restriction of $\omega$-consistency to the $\Sigma_n$-formulas. It is also shown that Rosser's Incompleteness Theorem does not generally hold for definable non-recursively enumerable theories; whence G\"odel-Rosser's Incompleteness Theorem is optimal in a sense. Though the proof of the incompleteness theorem using the $\Sigma_n$-soundness assumption is constructive, it is shown that there is no constructive proof for the incompleteness theorem using the $n$-consistency assumption, for $n\!>\!2$.

\bigskip

\centerline{${\backsim\!\backsim\!\backsim\!\backsim\!\backsim
\!\backsim\!\backsim\!
\backsim\!\backsim\!\backsim\!\backsim\!\backsim\!\backsim
\!\backsim\!
\backsim\!\backsim\!\backsim\!\backsim\!\backsim\!\backsim
\!\backsim\!
\backsim\!\backsim\!\backsim\!\backsim\!\backsim\!\backsim
\!\backsim\!
\backsim\!\backsim\!\backsim\!\backsim\!\backsim\!\backsim
\!\backsim\!
\backsim\!\backsim\!\backsim\!\backsim\!\backsim\!\backsim
\!\backsim\!
\backsim\!\backsim\!\backsim\!\backsim\!\backsim\!\backsim
\!\backsim\!
\backsim\!\backsim\!\backsim\!\backsim\!\backsim\!\backsim
\!\backsim\!
\backsim\!\backsim\!\backsim\!\backsim\!\backsim\!\backsim
\!\backsim\!
\backsim\!\backsim\!\backsim\!\backsim\!\backsim\!\backsim
\!\backsim\!
\backsim\!\backsim\!\backsim}$}

\bigskip



\bigskip

\bigskip

\bigskip

{\footnotesize
\noindent {\bf Acknowledgements} \
This is a part of the Ph.D. thesis of the second author written in the University of Tabriz  under the supervision of the first author who is partially supported by grant  $\textrm{N}^{\underline{\sf o}}$~93030033 from
  %the Institute for Research in Fundamental Sciences
  $\bigcirc\hspace{-2ex}{\not}\hspace{0.3ex}\bullet
  \hspace{-1.5ex}{\not}\hspace{-0.2ex}\bigcirc$
  $\mathbb{I}\mathbb{P}\mathbb{M}$.
}

\end{abstract}

\bigskip
\bigskip
\bigskip
\bigskip

\hspace{.75em} \fbox{\textsl{\footnotesize Date: 20 January  2016  (20.01.16)}}

\vfill

\bigskip
\noindent\underline{\centerline{}}
\centerline{\tt page 1 (of \pageref{LastPage})}

%%%
%%% the paper begins ...
%%%

\newpage
\setcounter{page}{2}
\SetWatermarkAngle{65}
\SetWatermarkLightness{0.925}
%\SetWatermarkFontSize{30cm}
\SetWatermarkScale{2.5}
\SetWatermarkText{\!\!\!\!\!\!\!\!\!\!\!\!\!\!\!\!\!\!\!
{\sc MANUSCRIPT (Submitted)}}

%%%%%%%%%%%%%%%%%%%%%%
%%%%%%%%%%%%%%%%%%%

\section{Introduction and Preliminaries}\label{sec-intro}
G\"odel's First Incompleteness Theorem is usually taken to be \textsl{the incompleteness of the first order theory of Peano Arithmetic} $\pa$. While  $\pa$ is not a complete theory, the theorem states much more than that. One of the most misleading ways for stating the theorem is:  \textsl{any sound theory containing} $\pa$ \textsl{is incomplete}, where a theory is called sound when all its axioms are true in the standard model of natural numbers $\mathbb{N}$. A quick counterexample for this statement, often asked by  new learners of the incompleteness, is that \textsl{but the theory of true arithmetic ${\rm Th}(\mathbb{N})$ is complete?!}, where ${\rm Th}(\mathbb{N})$ is the set of sentences that are true in the standard model of natural numbers. Of course, the obvious answer is that ${\rm Th}(\mathbb{N})$ {\sl is not recursively enumerable} ({\sc re} for short). So, the right rewording of G\"odel's First Incompleteness Theorem in its (weaker) semantic form is that
\textsl{any sound and {\sc re} theory containing} $\pa$  \textsl{is incomplete}.  Now, a natural second question is: \textsl{what about non-{\sc re} theories} (\textsl{that are sound and contain} $\pa$)? Again the same obvious answer shows up: \textsl{${\rm Th}(\mathbb{N})$ is not {\sc re}} (by the very theorem of G\"odel's first incompleteness) \textsl{and is complete}. So, the question of  the incompleteness of non-{\sc re} theories should come down to  more specific ones, at least to finitely representable theories, or, as the logicians say, definable ones. Hence, \textsl{do we have the incompleteness of definable theories} (\textsl{which are sound and contain} $\pa$)? This question has been answered affirmatively in the literature; see e.g. \cite{smullyan} or \cite{sereny04}.  G\"odel's original first incompleteness theorem did not assume the soundness of the theory in question, and he used the notion of $\omega$-consistency for that purpose. Later it was found out that the weaker notion of 1-consistency suffices for the theorem (see e.g. \cite{isaacson} or \cite{smith}).   By generalizing this equivalent notion to higher degrees ($\Pi_n$ in general) we will prove some generalizations of G\"odel's first incompleteness theorem for definable theories below.
Finally, Rosser's Trick proves G\"odel's result without assuming the 1-consistency of the theory. So, G\"odel-Rosser's Incompleteness Theorem, assuming  only the consistency of the theory, states  that \textsl{any consistent and {\sc re} theory containing} $\pa$  \textsl{is incomplete}. It is tempting to weaken the condition of recursive enumerability of the theory in this theorem; but we will see below that this is not possible. We can thus argue that G\"odel-Rosser's theorem is optimal in a sense.

%%%%%%%%%%%%%%%%%%
\subsection{Some Notation and Conventions}
We fix the following notation and conventions (mostly from  \cite{hajekpudlak,isaacson,kaye,smith,smullyan}).
Fix a language of arithmetic, like $\{0,S,+,\times,\leqslant\}$ (as in  \cite{hajekpudlak}) or $\{0,1,+,\times,<\}$ (as in  \cite{kaye}).
\begin{itemize}
\item For any natural number $n\in\mathbb{N}$ the term $\overline{n}$ represents this number in the fixed arithmetical language (which could be $S\cdots S(0)$ or $1+\cdots+1$ [$n$-times]). For a fixed G\"odel numbering of syntax, $\ulcorner\alpha\urcorner$ denotes the G\"odel number of the object $\alpha$; when there is no ambiguity we will write simply $\ulcorner\alpha\urcorner$  for the term $\overline{\ulcorner\alpha\urcorner}$.  Any G\"odel numbering consists of coding sequences; if $m$ is the code of a sequence, then the formula ${\sf Seq}(m)$ expresses this fact, and  its length is denoted by $\ell{\rm en}(m)$ and for any number $l\!<\!\ell{\rm en}(m)$ the $l^{\rm th}$ member of $m$ is denoted by $[m]_l$. A sequence $m$ is thus $\langle[m]_0,[m]_1,\cdots,[m]_{\ell{\rm en}(m)-1}\rangle$; and for any $k\!\leqslant\!\ell{\rm en}(m)$, the initial segment of $m$ with length $k$ is denoted by $\langle m\!\downharpoonright\!k\rangle$, that is $\langle[m]_0,[m]_1,\cdots,[m]_{k-1}\rangle$. Note that $\langle m\!\downharpoonright\!0\rangle=\emptyset$ and $\langle m\!\downharpoonright\!\ell{\rm en}(m)\rangle=m$. If $m$ is the G\"odel code of  a sentence, then ${\sf Sent}(m)$ expresses this fact.
    For a sequence of sentences like $m$, the formula ${\sf ConjSeq}(k,m)$ means that ``$k$ is the (G\"odel code of the) conjunction of all the members of $m$'', i.e., $k=\ulcorner\bigwedge\!\!\!\!\!\bigwedge_{i<\ell en(m)}\varphi_i\urcorner$ where $[m]_i=\ulcorner\varphi_i\urcorner$.
    The propositional connectives may act (as numeral partial functions) on natural numbers; for example $\neg m$ for $m\in\mathbb{N}$ is $\ulcorner\neg\alpha\urcorner$ where $m=\ulcorner\alpha\urcorner$, and for any $\circ\in\{\wedge,\vee,\rightarrow\}$ and $m,k\in\mathbb{N}$, $m\circ k=\ulcorner\alpha\circ\beta\urcorner$ where $m=\ulcorner\alpha\urcorner$ and $k=\ulcorner\beta\urcorner$.
    \item The classes of formulas $\{\Sigma_n\}_{n\in\mathbb{N}}$ and $\{\Pi_n\}_{n\in\mathbb{N}}$ are defined in the standard way \cite{hajekpudlak,kaye}: $\Sigma_0=\Pi_0$ is the class of bounded formulas (in which every universal quantifier has the form $\forall x ([x\leqslant t \rightarrow\cdots]$ and every existential quantifier has the form $\exists x [x\leqslant t\wedge\cdots]$), and the class $\Sigma_{n+1}$ contains the closure of $\Pi_n$ under the existential quantifiers, and is closed under disjunction, conjunction, existential quantifiers and bounded universal quantifiers; similarly,  the class $\Pi_{n+1}$ contains the closure of $\Sigma_n$ under the universal quantifiers, and  is closed under disjunction, conjunction, universal quantifiers and bounded existential quantifiers. By definition $\Delta_n=\Sigma_n\cap\Pi_n$. Let us note that the negation of a $\Sigma_n$-formula is a $\Pi_n$-formula, and vice versa; and that the formulas ${\sf Seq}(-)$,  ${\sf Sent}(-)$  and ${\sf ConjSeq}(-)$ can be taken to be $\Sigma_0$, and the functions $\ell{\rm en}(-)$, $[-]_{-}$ and $\langle -\!\downharpoonright\!-\rangle$ are definable by $\Sigma_0$-formulas.
\item The set of all true arithmetical formulas is denoted by ${\rm Th}(\mathbb{N})$; that is $\{\theta\!\in\!{\rm Sent}\mid\mathbb{N}\models\theta\}$. Similarly, for any $n$, $\Sigma_n$-${\rm Th}(\mathbb{N})=\{\theta\!\in\!\Sigma_n\text{-}{\rm Sent}\mid\mathbb{N}\models\theta\}$ and $\Pi_n$-${\rm Th}(\mathbb{N})=\{\theta\!\in\!\Pi_n\text{-}{\rm Sent}\mid\mathbb{N}\models\theta\}$. While by Tarski's Undefinability Theorem   the (G\"odel numbers of the members of the) set ${\rm Th}(\mathbb{N})$ is not definable, for $n\!>\!0$ the (G\"odel numbers of the members of the) set $\Sigma_n$-${\rm Th}(\mathbb{N})$ is definable by the $\Sigma_n$-formula $\Sigma_n$-${\sf True}(x)$ (stating that ``$x$ is the G\"odel number of a true $\Sigma_n$-sentence'') and the (G\"odel numbers of the members of the) set $\Pi_n$-${\rm Th}(\mathbb{N})$ is definable by the $\Pi_n$-formula $\Pi_n$-${\sf True}(x)$ (stating that ``$x$ is the G\"odel number of a true $\Pi_n$-sentence''). Robinson's Arithmetic is denoted by $\q$ which is a weak (induction-free) fragment of $\pa$.
\item  A definable theory is the  set of all logical consequences of a set  of sentences  that (the set of the G\"odel numbers of its members) is  definable by an arithmetical formula    ${\sf Axioms}_T(x)$ [meaning that $x$ is the G\"odel number of an axiom of $T$].  The formula ${\sf ConjAx}_T(x)$ states that ``$x$ is the G\"odel code of a formula which is a conjunction of some axioms of $T$'', i.e., $x=\ulcorner\bigwedge\!\!\!\!\!\bigwedge_{i=1}^\ell
    \varphi_i\urcorner$ where  $\bigwedge\!\!\!\!\!\bigwedge_{i=1}^\ell{\sf Axioms}_T(\ulcorner\varphi_i\urcorner)$.     %${\sf Seq}(x)\wedge\forall\ell\!<\!{\sf len}(x)\big[{\sf Axioms}_T([x]_\ell)\big]$.
    The proof predicate of first order logic is denoted by ${\sf Proof}(y,x)$ which is  a $\Sigma_0$-formula stating  that ``$y$ is the code of a proof of the formula with code $x$ in the first order logic''. So, for a definable theory $T$ the provability predicate of $T$ is the formula ${\sf Prov}_T(x)=\exists y,z\big[{\sf ConjAx}_T(z)\wedge{\sf Proof}(y,z\rightarrow x)\big]$; also the consistency predicate of  $T$ is  ${\sf Con}(T)=\neg{\sf Prov}_T(\ulcorner 0\neq 0\urcorner)$. Let us note that ${\sf Prov}_T$ defines the set of $T$-provable formulas, the deductive closure of  (the axioms of) $T$. For a class of formulas $\Gamma$ the theory $T$ is called $\Gamma$-definable when ${\sf Axioms}_T\in\Gamma$. Let us also note that if ${\sf Axioms}_T\!\in\!\Sigma_{n+1}$ or  ${\sf Axioms}_T\!\in\!\Pi_{n}$ then ${\sf ConjAx}_T\!\in\!\Sigma_{n+1}$ or ${\sf ConjAx}_T\!\in\!\Pi_{n}$, respectively, and so in either case ${\sf Prov}_T\!\in\!\Sigma_{n+1}$.
    %and so ${\sf Con}(T)\!\in\!\Pi_{n+1}$.

\item Theory $T$ decides the sentence $\varphi$ when either $T\vdash\varphi$ or $T\vdash\neg\varphi$. A theory is called complete when it can decide every sentence in its language.  A theory $T$ is called $\Gamma$-deciding when it can decide any sentence in  $\Gamma$. In the literature, a theory $T$ is called $\Gamma$-complete when for any sentence  $\varphi\in\Gamma$, if $\mathbb{N}\models\varphi$ then $T\vdash\varphi$. Note that if a sound theory is $\Gamma$-deciding then it is $\Gamma$-complete. A theory $T$ is called $\omega$-consistent when for no formula $\varphi$ both the conditions $(i)\; T\vdash\neg\varphi(\overline{n})$ for all $n\!\in\!\mathbb{N}$, and $(ii)\; T\vdash\exists x\varphi(x)$ hold together.  It is called $n$-consistent when for no formula $\varphi\!\in\!\Sigma_n$ with $\varphi=\exists x\psi(x)$ and $\psi\!\in\!\Pi_{n-1}$ one has $(i)\; T\vdash\neg\psi(\overline{n})$ for all $n\!\in\!\mathbb{N}$, and $(ii)\; T\vdash\varphi$. Theory $T$ is called    $\Gamma$-Sound, when for any sentence $\varphi\!\in\!\Gamma$, if $T\vdash\varphi$ then $\mathbb{N}\models\varphi$. For example, any consistent theory containing $\Pi_n\text{-}{\rm Th}(\mathbb{N})$ is $\Sigma_n$-sound. Let us note that, since ${\rm Th}(\mathbb{N})$ is a complete and thus a maximally consistent theory, the soundness of $T$ is equivalent to ${\rm Th}(\mathbb{N})\subseteq T$ and to the consistency of $T+{\rm Th}(\mathbb{N})$. In general, for any consistent extension $T$ of $\q$, the $\Sigma_n$-soundness of $T$ is equivalent to the consistency of $T+\Pi_n\text{-}{\rm Th}(\mathbb{N})$ %, that is ${\sf Con}\big(T+\Pi_n\text{-}{\rm Th}(\mathbb{N})\big)$
        (cf. Theorems~26,31 of~\cite{isaacson}). Also, for any $T\supseteq\q$, since $\q$ is a $\Sigma_1$-complete theory, the consistency of $T$ is equivalent to the consistency of $T+\Pi_0\text{-}{\rm Th}(\mathbb{N})$, i.e.  ${\sf Con}\big(T+\Pi_0\text{-}{\rm Th}(\mathbb{N})\big)$, which, in turn, is equivalent to the $\Sigma_0$-soundness of $T$ (cf. Theorem~5 of~\cite{isaacson}).
\end{itemize}
\vspace{-1em}
\begin{table}[h]
\begin{center}
  \begin{tabular}{|rcccl|}
  \hline
  {\footnotesize Semantic Condition} &  & {\footnotesize Conventional Notation}  &   & {\footnotesize Syntactic Condition} \\
  \hline
  \hline
    %\,\quad\,
    ({\small $\Sigma_\infty\!$})Soundness  of $T$ & $\equiv$ & $\mathbb{N}\models T$ & $\equiv$ &
    ${\sf Con}\big(T+$%\;\quad\;
  [{\small $\Pi_\infty$}]${\rm Th}(\mathbb{N})\big)$ \\
    \hline
     $\Sigma_n$-Soundness  of $T$ & $\equiv$ & ---------  & $\equiv$ & ${\sf Con}\big(T+\Pi_n\text{-}{\rm Th}(\mathbb{N})\big)$ \\
         \hline
     $\Sigma_1$-Soundness  of $T$ & $\equiv$ & $1\text{-}{\sf Con}(T)$ & $\equiv$ & ${\sf Con}\big(T+\Pi_1\text{-}{\rm Th}(\mathbb{N})\big)$ \\
          \hline
       $\Sigma_0$-Soundness  of $T$ & $\equiv$ & ${\sf Con}(T)$ & $\equiv$ & ${\sf Con}\big(T+\Pi_0\text{-}{\rm Th}(\mathbb{N})\big)$ \\
  \hline
\end{tabular}
\vspace{-1.5em}
\end{center}
\end{table}

%%%%%%%%%%%%%%%%%%%%%

%%%%%%%%%%%%%%%%%%%%

\subsection{Some Earlier Attempts and Results}
By G\"odel's incompleteness theorem, $\pa$ (and every {\sc re} extension of it) is not $\Pi_1$-complete; then what about $\textbf{S}=\pa+\Pi_{1}\text{-}{\rm Th}(\mathbb{N})$? Is this theory complete? For sure, it is $\Pi_1$-complete and $\Sigma_1$-complete; but can it be, say,   $\Pi_2$-complete? Let us note that $\textbf{S}$ is a $\Pi_1$-definable theory; i.e. ${\sf Axioms}_{\textbf{S}}\in\Pi_1$, and so ${\sf Prov}_{\textbf{S}}\in\Sigma_2$. So, it is natural to ask if the incompleteness phenomena still hold for definable arithmetical theories.

\subsubsection{Results of Jeroslow (1975)}
Jeroslow \cite{jeroslow} showed in 1975 that when the set of theorems of a consistent theory that contains $\pa$ is $\Delta_2$-definable, then it cannot contain the set of all true $\Pi_1$-sentences.

\bigskip
\centerline{
\fbox{$\textrm{Jeroslow (1975)}:\qquad  \pa\subseteq T  \;\; \& \;\;  {\sf Prov}_T\!\in\!\Delta_2 \;\; \& \;\; {\sf Con}(T) \;\;\; \Longrightarrow \;\;\; \Pi_1\text{-}{\rm Th}(\mathbb{N})\not\subseteq T$}}
\bigskip

\noindent
This result casts a new light on a classical theorem on the existence of a $\Delta_2$-definable complete extension of $\pa$ (see \cite{smorynski}): no such complete extension can contain all the true $\Pi_1$-sentences. Note that one cannot weaken the assumption ${\sf Prov}_T\!\in\!\Delta_2$ in the theorem, to,  say, ${\sf Prov}_T\!\in\!\Sigma_2$ because e.g. for the theory $\textbf{S}$ above we have ${\sf Prov}_\textbf{S}\!\in\!\Sigma_2$ and $\Pi_1\text{-}{\rm Th}(\mathbb{N})\subseteq\textbf{S}$.

\subsubsection{Results of H\'ajek (1977)}
Jeroslow's theorem was  generalized by H\'ajek (\cite{hajek}) who showed that
when the set of theorems of a consistent theory that contains $\pa$ is $\Delta_n$-definable, then it cannot be $\Pi_{n-1}$-complete:

\bigskip
\centerline{
\fbox{$\textrm{H\'ajek (1977)}:\qquad  \pa\subseteq T  \;\; \& \;\;  {\sf Prov}_T\!\in\!\Delta_n \;\; \& \;\; {\sf Con}(T) \;\;\; \Longrightarrow \;\;\; \Pi_{n-1}\text{-}{\rm Th}(\mathbb{N})\not\subseteq T$}}
\bigskip

\noindent Another result of H\'ajek (\cite{hajek}) is that if a deductively closed extension of $\pa$ is $\Pi_n$-definable and $n$-consistent, then it cannot be $\Pi_{n-1}$-complete:

\bigskip
\centerline{
\fbox{$\textrm{H\'ajek (1977a)}:\qquad  \pa\subseteq T  \;\; \& \;\;  {\sf Prov}_T\!\in\!\Pi_n \;\; \& \;\; n\text{--}{\sf Con}(T) \;\;\; \Longrightarrow \;\;\; \Pi_{n-1}\text{-}{\rm Th}(\mathbb{N})\not\subseteq T$}}
\bigskip

\noindent He also showed that no such theory can be complete; i.e., when $\pa\subseteq T$ \& ${\sf Prov}_T\!\in\!\Pi_n$ \&  $n\text{--}{\sf Con}(T)$ then $T$ is incomplete (indeed, a $\Pi_n$-sentence is independent from $T$). Here, we generalize this theorem by showing the existence of an independent   $\Pi_{n-1}$-sentence:

\bigskip
\centerline{
\fbox{$\textrm{{\bf Corollary~\ref{cor-hajek}}}:\qquad  \pa\subseteq T  \;\; \& \;\;  {\sf Prov}_T\!\in\!\Pi_n \;\; \& \;\; n\text{--}{\sf Con}(T) \;\;\; \Longrightarrow \;\;\; T\not\in\Pi_{n-1}\text{-}\textrm{Deciding}$}}
\bigskip

\begin{remark}[On the Proof of Theorem~2.5 in \cite{hajek}]{\rm
In Theorem~2.5 of \cite{hajek}  an  $n$-consistent theory $T$   is assumed to contain Peano Arithmetic (and be closed under deduction)  and its set of theorems  is assumed to be   $\Pi_n$-definable  for some $n\geqslant 2$. Then it is shown that    (1) $\Pi_{n-1}\text{-}{\rm Th}(\mathbb{N})\not\subseteq T$,
 and a proof is presented for the fact that
(2) $T$ is incomplete.}

\noindent {\rm In the proof of (1) for the sake of contradiction it is assumed that $\Pi_{n-2}\text{-}{\rm Th}(\mathbb{N})\subseteq T$; and at the end of the proof of (2) the inconsistency of $T$ has been inferred from the $T$-provability of  a false $\Pi_{n-2}$-sentence (denoted by $\tau_1(\overline{p},\overline{m},\overline{\varphi})$ in~\cite{hajek}). Of course, when $\Pi_{n-2}\text{-}{\rm Th}(\mathbb{N})\subseteq T$ then no false $\Pi_{n-2}$-sentence is provable in $T$. Probably, the proof did not intend to make use of the (wrong) assumption (of $\Pi_{n-2}\text{-}{\rm Th}(\mathbb{N})\subseteq T$); rather the intention could have been
 using  the completeness and $n$-consistency of $T$ to show that $T$ cannot prove any false $\Pi_{n-2}$-sentence. This is the subject of the next lemma (\ref{lem-hajek}) which fills an inessential minor gap in the proof of   \cite[Theorem~2.5]{hajek}.
}\hfill\ding{71}\end{remark}

%
%\begin{remark}{\rm
%In Theorem~2.5 of \cite{hajek} it is assumed that the $n$-consistent theory $T$  (which is closed under deduction) contains $\pa$ and (in our notation)  ${\sf Prov}_T\in\Pi_n$, for some $n\geqslant 2$. Then it is shown that
%
%(1) $\Pi_{n-1}\text{-}{\rm Th}(\mathbb{N})\not\subseteq T$,
%
%\noindent and a proof is presented for the fact that
%
%(2) $T$ is incomplete.
%
%\noindent In the proof of (1) for the sake of contradiction it is assumed that $\Pi_{n-2}\text{-}{\rm Th}(\mathbb{N})\subseteq T$; and at the end of the proof of (2) the inconsistency of $T$ has been inferred from the $T$-provability of  a false $\Pi_{n-2}$-sentence (denoted by $\tau_1(\overline{p},\overline{m},\overline{\varphi})$ in~\cite{hajek}). Of course, when $\Pi_{n-2}\text{-}{\rm Th}(\mathbb{N})\subseteq T$ then $T$ cannot prove a false $\Pi_{n-2}$-sentence. Probably, H\'ajek's proof was not intended to make use of the (wrong) assumption (of $\Pi_{n-2}\text{-}{\rm Th}(\mathbb{N})\subseteq T$); rather the intention could have been
% using  the completeness and $n$-consistency of $T$ to show that $T$ cannot prove any false $\Pi_{n-2}$-sentence. This is the subject of the next lemma (\ref{lem-hajek}).
%}\hfill\ding{71}\end{remark}

The following lemma generalizes Theorem~20 of \cite{isaacson} which states  that {\sl the true arithmetic} ${\rm Th}(\mathbb{N})$ {\sl is the only $\omega$-consistent extension of} $\pa$ (indeed $\q$)  {\sl that is complete}.

\begin{lemma}[A Gap in the Proof of  Theorem~2.5(2) in \cite{hajek}]\label{lem-hajek}
Any $n$-consistent and $\Pi_n$-deciding extension of {\rm $\q$} is $\Pi_n$-complete.
\end{lemma}

\begin{proof}
By induction on $n$. For $n=0$ there is nothing to prove. If the theorem holds for $n$ then we prove it for $n+1$ as follows. If $T$ is $(n+1)$-consistent and $\Pi_{n+1}$-deciding, but not $\Pi_{n+1}$-complete, there must exist some $\psi\!\in\!\Pi_{n+1}\text{-}{\rm Th}(\mathbb{N})$ such that $T\not\vdash\psi$. Write $\psi=\forall z\eta(z)$ for some $\eta\!\in\!\Sigma_n$; then  $\mathbb{N}\models\eta(m)$ for any $m\!\in\!\mathbb{N}$.
By the induction hypothesis, $T$ is $\Pi_n$-complete and so $\Sigma_n$-complete; thus $T\vdash\eta(\overline{m})$ for all $m\!\in\!\mathbb{N}$. On the other hand since $T$ is $\Pi_{n+1}$-deciding and  $T\not\vdash\psi$ we must have $T\vdash\neg\psi$, thus  $T\vdash\exists z\neg\eta(z)$. This contradicts the $(n+1)$-consistency of $T$.
\end{proof}

\begin{corollary}[Generalizing Theorem~2.5(2) of \cite{hajek}]\label{cor-hajek}
If the deductive closure of an $n$-consistent extension of {\rm $\pa$} is $\Pi_n$-definable, then it has an independent $\Pi_{n-1}$-sentence (for any $n\geqslant 2$).
\end{corollary}
%\medskip

\begin{proof}
If for a theory $T$ we have $\pa\subseteq T$ and  ${\sf Prov}_T\!\in\!\Pi_n$ and   $n\text{--}{\sf Con}(T)$ then it cannot be $\Pi_{n-1}$-deciding, since otherwise by Lemma~\ref{lem-hajek} (and $(n-1)$-consistency of $T$),
     $\Pi_{n-1}\text{-}{\rm Th}(\mathbb{N})\subseteq T$; this is in contradiction with  Theorem~2.5(1) of \cite{hajek} which states that $\Pi_{n-1}\text{-}{\rm Th}(\mathbb{N})\not\subseteq T$ under the above assumptions.
\end{proof}

Below we will give yet another generalization of the above corollary (and a result of \cite{hajek}) in Corollary~\ref{cor-hajek2}. H\'ajek \cite{hajek} has also showed  that if the set of axioms of a consistent theory  is $\Pi_1$-definable and that theory contains $\pa$ and all the true $\Pi_1$-sentences, then it is not $\Pi_2$-deciding. In Corollary~\ref{cor-haj2} we will generalize this result by showing that no consistent $\Pi_n$-definable and $\Pi_n$-complete extension of $\q$ is  $\Pi_{n+1}$-deciding.

\medskip
\centerline{
\fbox{$\textrm{{\bf Corollary~\ref{cor-haj2}}}:\qquad  \q\subseteq T  \;\; \& \;\;  {\sf Axioms}_T\!\in\!\Pi_n \;\; \& \;\; {\sf Con}(T) \;\; \& \;\; \Pi_n\text{-}{\rm Th}(\mathbb{N})\subseteq T \;\;\; \Longrightarrow \;\;\; T\not\in\Pi_{n+1}\text{-}\textrm{Deciding}$}}
%\medskip

\subsubsection{Some Recent Attempts}
In the result  of Jeroslow (Theorem~2 of~\cite{jeroslow}) and H\'ajek's generalizations (Theorem~2.5(1) and Theorem~2.8 of~\cite{hajek}) there is no incompleteness; we have only some non-inclusion (of the set of true $\Pi_1$ or $\Pi_n$ sentences in the theory). In the incompleteness theorems of H\'ajek (\cite{hajek} and Corollaries~\ref{cor-hajek} and~\ref{cor-haj2}) we had the, somewhat strong, assumptions of $n$-consistency or $\Pi_n$-completeness (with consistency). It is natural to ask if we can weaken these assumptions (like in Rosser's Trick) to mere consistency;
and some attempts  \cite{kitada09,ishii03} have been made in this direction. Let us note that Rosserian (also G\"odelean) proofs make sense for definable theories only (for example the undefinable theory ${\rm Th}(\mathbb{N})$ is complete) for the reason that when a theory $T$ is definable one can construct its provability predicate ${\sf Prov}_T$, and once one has a provability predicate for $T$ then it becomes a definable theory.

 We note that the proofs of G\"odel-Rosser's incompleteness theorem for non-{\sc re} theories  given in \cite{kitada09,ishii03} are both wrong; for the falsity of the argument of \cite{kitada09} one can see \cite{salehi-review1}; cf. also  \cite{kitada11} and \cite{salehi-review2}. The falsity of the proof of \cite{ishii03} is shown in the following remark (\ref{rem-2}). Unfortunately, there is no hope of extending G\"odel-Rosser's incompleteness theorem to definable theories, even to $\Pi_1$-definable ones; our Corollary~\ref{cor-rosser} below shows that even a (consistent and) $\Pi_1$-definable theory (extending $\q$) can be complete. This clashes all the hopes for a general incompleteness  phenomenon in the class  of definable, and consistent, theories.

\begin{remark}%[An Earlier Attempt]
\label{rem-2}{\rm
Unfortunately, the proof of G\"odel-Rosser's incompleteness theorem for non-{\sc re} theories  given in \cite{ishii03} is wrong: In the proof of Lemma~3 in \cite{ishii03} the author uses the  Diagonal Lemma for  $\neg F(x)$, where $F$ is constructed in Lemma~S (Chapter~VI) of \cite{smullyan} (together with its Lemma~2 in Chapter~V); it can be seen that $F\!\in\!\Pi_1$ and so $A\!\in\!\Sigma_1$. If, as claimed in Lemma~3 (and Theorem and Corollary) of \cite{ishii03}, for any $\Pi_1$-definable consistent extension of $\q$ there existed a $\Sigma_1$-sentence $A$ independent from it, then the theory  $\q+\Pi_1$-${\rm Th}(\mathbb{N})$ would have had a $\Pi_1$-sentence independent from it. But it is well-known that this theory is $\Sigma_1$-complete and $\Pi_1$-complete. So, the proof of the main theorem of \cite{ishii03} is flawed.  In fact, the mistaken step is in the proof of Lemma~2 where the author claims that ``$m_1$ can be chosen such that $m_2\leqslant m_1$ and hence $\overline{R}(k,m_2,{\rm Neg}(n))$.'' But if we choose $m_1$ arbitrarily large then the condition $\forall x\!\leqslant\!m_1\overline{R}\big(k,x,{\rm Neg}(n)\big)$ may not necessarily hold anymore. Indeed, Theorem~\ref{thm-rosser} for $n\!=\!0$ is the negation of what is claimed in the Abstract of   \cite{ishii03}.
}\hfill\ding{71}\end{remark}

%
%\begin{remark}\label{rem-2}{\rm
%In the proof of Lemma~3 in \cite{ishii03} the author uses (G\"odel-Carnap's)\footnote{This lemma was used by G\"odel for  the unprovability formula, and was noticed by Carnap to hold for all formulas.} Diagonal Lemma for  $\neg F(x)$, where $F$ is constructed in Lemma~S (Chapter~VI) of \cite{smullyan} (together with its Lemma~2 in Chapter~V); it can be seen that $F\in\Pi_1$ and so $A\in\Sigma_1$. If, as claimed in Lemma~3 and Theorem and Corollary of \cite{ishii03}, for any $\Pi_1$-definable consistent extension of $\q$ there existed a $\Sigma_1$-sentence $A$ independent from it, then  $\q+\Pi_1$-${\rm Th}(\mathbb{N})$ would have had a $\Pi_1$-sentence independent from it. But it is well-known that this theory is $\Sigma_1$-complete and $\Pi_1$-complete. So, the proof of the main theorem of \cite{ishii03} is flawed.  In fact, the mistaken step is in the proof of Lemma~2 where the author claims that ``$m_1$ can be chosen such that $m_2\leqslant m_1$ and hence $\overline{R}(k,m_2,{\rm Neg}(n))$.'' But if we choose $m_1$ arbitrarily large then the condition $\forall x\!\leqslant\!m_1\overline{R}\big(k,x,{\rm Neg}(n)\big)$ may not necessarily hold anymore.
%}\hfill\ding{71}\end{remark}
%

%%%%%%%%%%%%%%%%%%%%
%%%%%%%%%%%%%%%%%%%
\section{G\"odel's Theorem Generalized}\label{sec-godel}

\subsection{Semantic Form of G\"odel's Theorem}
G\"odel's First Incompleteness Theorem in its (weaker) semantic form states that no sound and {\sc re} extension of $\q$ can be $\Pi_1$-complete. Noting that a set is {\sc re} if and only if it is $\Sigma_1$-definable, this theorem can be depicted as:

\medskip
\centerline{
\fbox{$\textrm{G\"odel's 1}^{\rm st} \; \textrm{\footnotesize (Semantic)}:\qquad  \q\subseteq T  \;\; \& \;\;  {\sf Axioms}_T\!\in\!\Sigma_1 \;\; \& \;\; \mathbb{N}\models T \;\;\; \Longrightarrow \;\;\; \Pi_1\text{-}{\rm Th}(\mathbb{N})\not\subseteq T$}}
\medskip

\noindent A natural generalization of this theorem is the following (cf. Chapter~III of \cite{smullyan}, or  Corollary~1 of \cite{sereny04}):
\begin{theorem}\label{prop-godel}
No sound and $\Sigma_n$-definable  ($n\!>\!0$) extension of {\rm $\q$} can be $\Pi_n$-complete.
\end{theorem}
\centerline{
\fbox{$\textrm{{\bf Theorem~\ref{prop-godel}}}:\qquad  \q\subseteq T  \;\; \& \;\;  {\sf Axioms}_T\!\in\!\Sigma_n \;\; \& \;\; \mathbb{N}\models T \;\;\; \Longrightarrow \;\;\; \Pi_n\text{-}{\rm Th}(\mathbb{N})\not\subseteq T$}}
\medskip
\begin{proof}
Suppose $T$ is a sound extension of $\q$ such that ${\sf Axioms}_T\in\Sigma_n$. By Diagonal Lemma (see e.g.~\cite{hajekpudlak,smith}) there exists a sentence $\boldsymbol\gamma$ such that ${\tt Q}\vdash \boldsymbol\gamma\longleftrightarrow\neg{\sf Prov}_T(\ulcorner\boldsymbol\gamma\urcorner)$.
Obviously, $\boldsymbol\gamma\in\Pi_n$. We show that $(\dag)\; \mathbb{N}\models\boldsymbol\gamma$.

Since, otherwise (if $\mathbb{N}\models\neg\boldsymbol\gamma$ then) there must exist some  $k,m\in\mathbb{N}$ such that $\mathbb{N}\models{\sf ConjAx}_T(k)$ and $\mathbb{N}\models{\sf Proof}(m,k\!\rightarrow\!\ulcorner\boldsymbol\gamma\urcorner)$. Whence, $T\vdash\boldsymbol\gamma$ which contradicts the soundness of $T$. So, %we have shown that
$\mathbb{N}\models\boldsymbol\gamma$. Now, we show that $T\not\vdash\boldsymbol\gamma$.  For the sake of contradiction, assume $T\vdash\boldsymbol\gamma$. Then, by the  compactness theorem, there are some $\varphi_1,\cdots,\varphi_l$ such that $\mathbb{N}\models\bigwedge\!\!\!\!\!\bigwedge_{i=1}^l{\sf Axioms}_T(\ulcorner\varphi_i\urcorner)$ and $\vdash\bigwedge\!\!\!\!\!\bigwedge_{i=1}^l
\varphi_i\rightarrow\boldsymbol\gamma$. If $m$ is the code of this proof and $k$ is the code of $\bigwedge\!\!\!\!\!\bigwedge_{i=1}^l\varphi_i$ then $\mathbb{N}\models{\sf ConjAx}_T(k)\wedge{\sf Proof}
(m,k\!\rightarrow\!\ulcorner\boldsymbol\gamma\urcorner)$, or in other words $\mathbb{N}\models{\sf Prov}_T(\ulcorner\boldsymbol\gamma\urcorner)$ so  $\mathbb{N}\models\neg\boldsymbol\gamma$ contradicting $(\dag)$ above. Thus, $\boldsymbol\gamma\in\Pi_n\text{-}{\rm Th}(\mathbb{N})\setminus T$.
\end{proof}

\begin{corollary}\label{cor-godel2}
No sound and $\Pi_n$-definable extension of {\rm $\q$} can be $\Pi_{n+1}$-complete.
\end{corollary}
\centerline{
\fbox{$\textrm{{\bf Corollary~\ref{cor-godel2}}}:\qquad  \q\subseteq T  \;\; \& \;\;  {\sf Axioms}_T\!\in\!\Pi_n \;\; \& \;\; \mathbb{N}\models T \;\;\; \Longrightarrow \;\;\; \Pi_{n+1}\text{-}{\rm Th}(\mathbb{N})\not\subseteq T$}}
\medskip
\begin{proof}
It suffices to note that any $\Pi_n$-definable is also $\Sigma_{n+1}$-definable.
\end{proof}

\begin{remark}\label{rem-1}{\rm
It is well known that $\q$ is $\Sigma_1$-complete (see e.g. \cite{hajekpudlak,kaye,smith}) but not $\Pi_1$-complete (by G\"odel's first incompleteness theorem, see  e.g. \cite{hajekpudlak,kaye,smith}). So, $\Sigma_1$-completeness does not imply $\Pi_1$-completeness, and in general, $\Sigma_n$-completeness does not imply $\Pi_n$-completeness, since for example the $\Sigma_n$-complete and sound theory  $\q+\Sigma_n$-${\rm Th}(\mathbb{N})$ is not $\Pi_n$-complete by Theorem~\ref{prop-godel}. On the other hand, $\Pi_n$-completeness (of any theory $T$) implies (its) $\Sigma_n$-completeness, even (its) $\Sigma_{n+1}$-completeness: for any true $\Sigma_{n+1}$-sentence  $\exists x_1,\ldots,x_k\theta(x_1,\ldots,x_k)$ with $\theta\in\Pi_n$   there are $n_1,\ldots,n_k\in\mathbb{N}$ such that $\mathbb{N}\models\theta(n_1,\ldots,n_k)$, and so by $\Pi_n$-completeness of $T$ we have $T\vdash\theta(\overline{n_1},\ldots,\overline{n_k})$ whence $T\vdash\exists x_1,\ldots,x_k\theta(x_1,\ldots,x_k)$. In symbols: \fbox{$\Pi_n\text{-}{\rm Th}(\mathbb{N})\subseteq T \;\Longrightarrow\; \Sigma_{n+1}\text{-}{\rm Th}(\mathbb{N})\subseteq T$} (cf.~\cite[Lemma~2.2]{hajek}).
}\hfill\ding{71}\end{remark}

\subsection{General Form of G\"odel's Theorem}
The original form of G\"odel's first incompleteness theorem states that a recursively enumerable extension of $\q$ which is $\omega$-consistent cannot be $\Pi_1$-deciding. This syntactic notion was introduced to take place of the semantic notion of soundness. Later it was found out that G\"odel's proof works with the weaker assumption of 1-consistency which is equivalent to the consistency (of the theory) with $\Pi_1$-${\rm Th}(\mathbb{N})$ (see \cite{isaacson}):

\medskip
\centerline{
\fbox{$\textrm{G\"odel's 1}^{\rm st}\; (1931):\qquad  \q\subseteq T  \;\; \& \;\;  {\sf Axioms}_T\!\in\!\Sigma_{1} \;\; \& \;\; {\sf Con}\big(T+\Pi_1$-${\rm Th}(\mathbb{N})\big) \;\;\; \Longrightarrow \;\;\; T\not\in\Pi_1\text{-}\textrm{Deciding}$}}
\medskip

\noindent
A natural generalization of the theorem in this form is the $\Pi_n$-undecidability  of any $\Sigma_n$-definable extension of $\q$ which is consistent with $\Pi_n$-${\rm Th}(\mathbb{N})$; proved in Corollary~\ref{cor-1st} of the following theorem.

\begin{theorem}\label{thm-1st}
No $\Pi_n$-definable extension of  {\rm $\q$}  can be $\Pi_{n+1}$-deciding  if it is consistent with $\Pi_n$-${\rm Th}(\mathbb{N})$.
\end{theorem}
\centerline{
\fbox{$\textrm{{\bf Theorem~\ref{thm-1st}}}:\qquad  \q\subseteq T  \;\; \& \;\;  {\sf Axioms}_T\!\in\!\Pi_{n} \;\; \& \;\; {\sf Con}\big(T+\Pi_n\textrm{-}{\rm Th}(\mathbb{N})\big) \;\;\; \Longrightarrow \;\;\; T\not\in\Pi_{n+1}\text{-}\textrm{Deciding}$}}
\medskip
\begin{proof}
By  Diagonal Lemma there exists a sentence $\boldsymbol\gamma$ such that

 \noindent $\q\vdash\boldsymbol\gamma\longleftrightarrow\forall u,z\Big(\exists x,y\!\leqslant\!u\big[\langle x,y\rangle=u\wedge\Pi_n\text{-}{\sf True}(x)\wedge{\sf ConjAx}_T(y)\wedge{\sf Proof}(z,x\!\wedge\!y\!\rightarrow\!\ulcorner \boldsymbol\gamma\urcorner)\big]\rightarrow$

 $\exists u'\!\leqslant\!u\exists z'\!\leqslant\!z\big(\exists x',y'\!\leqslant\!u'\big[\langle x',y'\rangle=u'\wedge\Pi_n$-${\sf True}(x')\wedge{\sf ConjAx}_T(y')\wedge{\sf Proof}(z',x'\!\wedge\!y'\!\rightarrow\!\ulcorner \neg\boldsymbol\gamma\urcorner)\big]\big)\Big)$\hfill $(\star)$

where, $\langle-,-\rangle$ is an injective pairing (such as $\langle u,v\rangle=(u+v)^2+u$).

\noindent Obviously, $\boldsymbol\gamma\in\Pi_{n+1}$. We show that $\boldsymbol\gamma$ is independent from $T^\ast=T+\Pi_n$-${\rm Th}(\mathbb{N})$.

\noindent Put $\Psi(u,z)=\exists x,y\!\leqslant\!u\big[\langle x,y\rangle=u\wedge\Pi_n\text{-}{\sf True}(x)\wedge{\sf ConjAx}_T(y)\wedge{\sf Proof}(z,x\!\wedge\!y\!\rightarrow\!\ulcorner \boldsymbol\gamma\urcorner)\big]$ and

$\;\,\widehat{\Psi}(u,z)=\exists x,y\!\leqslant\!u\big[\langle x,y\rangle=u\wedge\Pi_n\text{-}{\sf True}(x)\wedge{\sf ConjAx}_T(y)\wedge{\sf Proof}(z,x\!\wedge\!y\!\rightarrow\!\ulcorner
\neg\boldsymbol\gamma\urcorner)\big]$.

\noindent Thus, $(\star)$  is now translated to

\qquad\qquad\qquad  {$\q\vdash\boldsymbol\gamma\longleftrightarrow\forall u,z\big[\Psi(u,z)\rightarrow\exists u'\!\leqslant\!u\exists z'\!\leqslant\!z\widehat{\Psi}(u',z')\big]$.}\hfill $(\star ')$

\begin{itemize}\itemindent=3.5em
\item[$(T^\ast\not\vdash\boldsymbol\gamma)$:\,\;\,] If $T^\ast\vdash\boldsymbol\gamma$ then there are $\psi\!\in\!\Pi_n\text{-}{\rm Th}(\mathbb{N})$ (note that $\Pi_n\text{-}{\rm Th}(\mathbb{N})$ is closed under conjunction) and a conjunction $\varphi$ of the axioms of $T$  such that $\vdash\psi\wedge\varphi\rightarrow\boldsymbol\gamma$. Let $m$ be the G\"odel code of this proof and let $k=\langle\ulcorner\psi\urcorner,\ulcorner\varphi\urcorner
    \rangle$. Now, we have $\mathbb{N}\models\Psi(k,m)$, and so by the $\Pi_n$-completeness of $T^\ast$ we have  $T^\ast\vdash\Psi(\overline{k},\overline{m})$, thus by $(\star ')$, $T^\ast\vdash\exists u'\!\leqslant\!\overline{k}\exists z'\!\leqslant\!\overline{m}\widehat{\Psi}(u',z')$.\qquad $(\ddag)$

 On the other hand by the consistency of $T^\ast$ we have    $T^\ast\not\vdash\neg\boldsymbol\gamma$. So, for any  $q=\langle q_1,q_2\rangle,r\in\mathbb{N}$ we have that if $\mathbb{N}\models\Pi_n\text{-}{\sf True}(q_1)\wedge{\sf ConjAx}_T(q_2)$ then $\mathbb{N}\models\neg{\sf Proof}(r,q_1\!\wedge\!q_2\!\rightarrow\!\ulcorner
 \neg\boldsymbol\gamma\urcorner)$. Whence,   $\mathbb{N}\models\neg\widehat{\Psi}(q,r)$ holds for all $q,r\in\mathbb{N}$  in particular for all $q\!\leqslant\!k, r\!\leqslant\!m$; thus
${\mathbb{N}\models\forall u'\!\leqslant\!k\forall z'\!\leqslant\!m\neg\widehat{\Psi}(u',z').}$
 Now, $\forall u'\!\leqslant\!k\forall z'\!\leqslant\!m\neg\widehat{\Psi}(u',z')$ is a true $\Sigma_n$-sentence and $T^\ast$ is a $\Pi_n$-complete theory; so  by Remark~\ref{rem-1},   $T^\ast\vdash\forall u'\!\leqslant\!\overline{k}\forall z'\!\leqslant\!\overline{m}\neg\widehat{\Psi}(u',z')$  contradicting $(\ddag)$ above!

\item[$(T^\ast\not\vdash\neg\boldsymbol\gamma)$:] If $T^\ast\vdash\neg\boldsymbol\gamma$ then from $(\star ')$ it follows that
    \newline\centerline{$(i)\qquad  T^\ast\vdash\exists u,z\big[\Psi(u,z)\wedge\forall u'\!\leqslant\!u\forall z'\!\leqslant\!z\neg\widehat{\Psi}(u',z')\big].$}
     By   the compactness theorem (applied to the deduction $T^\ast\vdash\neg\boldsymbol\gamma$) there are $k=\langle k_1,k_2\rangle,m\!\in\!\mathbb{N}$ such that $\mathbb{N}\models\Pi_n\text{-}{\sf True}(k_1)\wedge{\sf ConjAx}_T(k_2)\wedge{\sf Proof}(m,k_1\!\wedge\!k_2\!\rightarrow\!\ulcorner
     \neg\boldsymbol\gamma\urcorner)$.
     Below,  we will show that
      \newline\centerline{$(ii)\qquad T^\ast\vdash\forall u,z \big[\neg\Psi(u,z)\vee\exists  u'\!\leqslant\!u\exists z'\!\leqslant\!z\widehat{\Psi}(u',z')\big],$} which contradicts $(i)$ above. The proof of $(ii)$ will be  in three steps:

    \qquad\qquad   (1)\; $T^\ast\vdash \forall u\!\geqslant\!\overline{k}\forall z\!\geqslant\!\overline{m}\big[\exists  u'\!\leqslant\!u\exists z'\!\leqslant\!z\widehat{\Psi}(u',z')\big]$

      \qquad\qquad   (2)\; $T^\ast\vdash \forall u\!<\!\overline{k}\forall z\big[\neg\Psi(u,z)\big]$

      \qquad\qquad   (3)\; $T^\ast\vdash \forall u\forall z\!<\!\overline{m}\big[\neg\Psi(u,z)\big]$
 \begin{itemize}%\itemindent=1em
 \item[(1)]
          Since  $\widehat{\Psi}(\overline{k},\overline{m})=\exists x,y\!\leqslant\!\overline{k}\big[\langle x,y\rangle=\overline{k}\wedge\Pi_n\text{-}{\sf True}(x)\wedge{\sf ConjAx}_T(y)\wedge{\sf Proof}(\overline{m},x\!\wedge\!y\!\rightarrow\!\ulcorner
          \neg\boldsymbol\gamma\urcorner)
          \big]$  is a true $\Pi_n$-sentence (for $x=k_1,y=k_2$), then $T^\ast$ proves it, so
(1) holds (for $u'=\overline{k},z'=\overline{m}$).
\item[(2)] It suffices to show $T^\ast\vdash\forall z\neg\Psi(\overline{i},z)$ for all $i<k$. Fix an $i<k$. If there are no $i_1,i_2$ such that $\langle i_1,i_2\rangle=i$ then $T^\ast\vdash\forall z\neg\Psi(\overline{i},z)$ holds trivially; otherwise, fix $i_1,i_2$ with $\langle i_1,i_2\rangle=i$. If either $\neg\Pi_n\text{-}{\sf True}(i_1)$ or $\neg{\sf ConjAx}_T(i_2)$, then again $T^\ast\vdash\forall z\neg\Psi(\overline{i},z)$ holds. Finally, assume that $\Pi_n\text{-}{\sf True}(i_1)\wedge{\sf ConjAx}_T(i_2)$ is true. Then, by the consistency of $T^\ast$ we have $T^\ast\not\vdash\boldsymbol\gamma$ and so for all $p\in\mathbb{N}$ we have $\mathbb{N}\models\neg{\sf Proof}(p,i_1\!\wedge\!i_2\!\rightarrow\!\ulcorner
    \boldsymbol\gamma\urcorner)$. Whence, $T^\ast$ proves the true $\Pi_1$-sentence $\forall z\neg{\sf Proof}(z,i_1\!\wedge\!i_2\!\rightarrow\!\ulcorner
    \boldsymbol\gamma\urcorner)$, and so $T^\ast\vdash\forall z\neg\Psi(\overline{i},z)$.
\item[(3)] Again we need to show $T^\ast\vdash \forall u\big[\neg\Psi(u,\overline{j})\big]$ for all $j<m$. Since $T^\ast$ proves the true $\Pi_1$-sentence $\forall x,y,v,w\big[{\sf Proof}(w,x\!\wedge\!y\!\rightarrow\!v)\rightarrow\langle x,y\rangle<w\big]$ then $T^\ast\vdash\forall u\big[\Psi(u,\overline{j})\rightarrow u<\overline{j}\big]$. Since, by an argument similar to that of (2) above, we can show that $T^\ast\vdash\forall u<\overline{j}\big[\neg\Psi(u,\overline{j})\big]$, then $T^\ast\vdash \forall u\forall z\!<\!\overline{m}\big[\neg\Psi(u,z)\big]$ holds too.
\end{itemize}
\end{itemize}
Whence, $T^\ast$, and so $T$, is not $\Pi_{n+1}$-deciding. Let us note that the above proof also shows that $\mathbb{N}\models\boldsymbol\gamma$.
\end{proof}

\noindent Note that Theorem~\ref{thm-1st} is Rosser's Theorem  for $n=0$, and indeed one can feel that the above, rather long, proof is in spirit more Rosserian (than G\"odelean) in the sense that the proof uses somehow Rosser's Trick.

\begin{corollary}\label{cor-haj2}
No consistent $\Pi_n$-definable and $\Pi_n$-complete extension of {\rm $\q$} can be
$\Pi_{n+1}$-deciding.
\end{corollary}
\begin{proof}
If $T\supseteq\q+\Pi_n\text{-}{\rm Th}(\mathbb{N})$ is consistent and $\Pi_n$-definable, then $T+\Pi_n\text{-}{\rm Th}(\mathbb{N})$ is consistent, and so by Theorem~\ref{thm-1st}, $T$ is not $\Pi_{n+1}$-deciding.
\end{proof}

\begin{corollary}\label{cor-hajek2}
No $\Pi_n$-definable extension of  {\rm $\q$}  can be $\Pi_{n+1}$-deciding  if it is $n$-consistent.
\end{corollary}
\centerline{
\fbox{$\textrm{{\bf Corollary~\ref{cor-hajek2}}}:\qquad  \q\subseteq T  \;\; \& \;\;  {\sf Axioms}_T\!\in\!\Pi_{n} \;\; \& \;\; n\text{-}{\sf Con}(T) \;\;\; \Longrightarrow \;\;\; T\not\in\Pi_{n+1}\text{-}\textrm{Deciding}$}}
\medskip
\begin{proof}
Let  $T\supseteq\q$ be an $n$-consistent extension of $\q$  such that ${\sf Axioms}_T\in\Pi_{n}$. If $T$ is not $\Pi_n$-deciding, then there is nothing to prove. If $T$ is $\Pi_n$-deciding, then by Lemma~\ref{lem-hajek} we have $\Pi_n\text{-}{\rm Th}(\mathbb{N})\subseteq T$, and so $T$ is consistent with $\Pi_n\text{-}{\rm Th}(\mathbb{N})$. Thus, by Theorem~\ref{thm-1st}, $T$ is not $\Pi_{n+1}$-deciding.
\end{proof}

\noindent Let us note that   for a $\Pi_n$-definable extension of $\q$  (like $T$) Corollary~\ref{cor-hajek} implies the $\Pi_{n+1}$-undecidability (of $T$) under the condition of $(n+2)$-consistency (of $T$) because ${\sf Axioms}_T\in\Pi_n$ implies ${\sf Prov}_T\in\Pi_{n+2}$; while Corollary~\ref{cor-hajek2} derives the same conclusion (of the $\Pi_{n+1}$-undecidability of $T$) under the assumption of $n$-consistency (of $T$). So, we can argue that Theorem~\ref{thm-1st} somehow strengthens Theorem~2.5(2) of~\cite{hajek}.
The following lemma, needed later, generalizes (and modifies)  Craig's Trick.

\begin{lemma}\label{lem-craig}
Any $\Sigma_{n+1}$-definable (arithmetical) theory  is equivalent with a $\Pi_n$-definable theory.
\end{lemma}

\begin{proof}
If ${\sf Axioms}_T(x)=\exists x_1\cdots\exists x_n \theta(x,x_1,\cdots,x_n)$ with $\theta\!\in\!\Pi_{n}$ then ${\sf Axioms}_T(x)\equiv\exists y\theta'(x,y)$ with  $\theta'(x,y)=\exists x_1\!\leqslant\!y\cdots\exists x_n\!\leqslant\!y\,\theta(x,x_1,\cdots,x_n)\in\Pi_{n}$. Now,   $T'=\{\varphi\wedge\big(\overline{k}\!=\!\overline{k})
\mid\mathbb{N}
\models\theta'(\ulcorner\varphi\urcorner,k)\}$  is equivalent with  $T$ and is $\Pi_{n}$-definable by  ${\sf Axioms}_{T'}(x)\equiv\exists y,z\!\leqslant\!x\big(\theta'(y,z)\wedge
\big[x=(y\wedge\ulcorner\overline{z}\!=\!\overline{z}
\urcorner)\big]\big)$.
\end{proof}

\begin{corollary}\label{cor-1st}
No $\Sigma_n$-definable ($n\!>\!0$) extension of {\rm $\q$}  can be $\Pi_{n}$-deciding if it is consistent with $\Pi_n$-${\rm Th}(\mathbb{N})$.
\end{corollary}
\centerline{
\fbox{$\textrm{{\bf Corollary~\ref{cor-1st}}}:\qquad  \q\subseteq T  \;\; \& \;\;  {\sf Axioms}_T\!\in\!\Sigma_{n} \;\; \& \;\; {\sf Con}\big(T+\Pi_n\text{-}{\rm Th}(\mathbb{N})\big) \;\;\; \Longrightarrow \;\;\; T\not\in\Pi_{n}\text{-}\textrm{Deciding}$}}
\medskip
\begin{proof}
For $n=1$ this is G\"odel's first incompleteness theorem. Suppose that $n>1$, and that ${\sf Axioms}_T\in\Sigma_n$ for some $T\supseteq\q$ such that $T+\Pi_n$-${\rm Th}(\mathbb{N})$ is consistent. By Lemma~\ref{lem-craig} there exists  a $\Pi_{n-1}$-definable theory $T'$  equivalent with $T$.   Now, $T'$ contains $\q$, is $\Pi_{n-1}$-definable, and is consistent with $\Pi_{n-1}$-${\rm Th}(\mathbb{N})$ (because  $T$ is consistent with $\Pi_n$-${\rm Th}(\mathbb{N})$).  Thus,  by Theorem~\ref{thm-1st} the theory $T'$ is not $\Pi_n$-deciding; neither is  $T$.
\end{proof}

\noindent Actually, the consistency of $T$ with $\Pi_{n-1}$-${\rm Th}(\mathbb{N})$ suffices for the above proof to go through.

\begin{corollary}\label{cor-2nd}%[Cor. Title]
No $\Sigma_n$-definable %($n\!\!>\!\!1$)
 extension of {\rm $\q$}  can be $\Pi_{n}$-deciding if it is consistent with $\Pi_{n-1}$-${\rm Th}(\mathbb{N})$.
\hfill\ding{113}
\end{corollary}
\centerline{
\fbox{$\textrm{{\bf Corollary~\ref{cor-2nd}}}:\qquad  \q\subseteq T  \;\; \& \;\;  {\sf Axioms}_T\!\in\!\Sigma_{n} \;\; \& \;\; {\sf Con}\big(T+\Pi_{n-1}\text{-}{\rm Th}(\mathbb{N})\big) \;\;\; \Longrightarrow \;\;\; T\not\in\Pi_{n}\text{-}\textrm{Deciding}$}}
\medskip

\noindent By G\"odel's first incompleteness theorem no $1$-consistent and $\Sigma_1$-definable extension of $\q$ can be $\Pi_1$-deciding; another generalization of this theorem is the $\Pi_n$-undecidability of any $n$-consistent and $\Sigma_n$-definable extension of $\q$.

\begin{corollary}\label{cor-g2}
No $\Sigma_n$-definable %($n\!\!>\!\!1$)
 extension of {\rm $\q$}  can be $\Pi_{n}$-deciding if it is $n$-consistent.
\end{corollary}
\centerline{
\fbox{$\textrm{{\bf Corollary~\ref{cor-g2}}}:\qquad  \q\subseteq T  \;\; \& \;\;  {\sf Axioms}_T\!\in\!\Sigma_{n} \;\; \& \;\; n\text{-}{\sf Con}(T) \;\;\; \Longrightarrow \;\;\; T\not\in\Pi_{n}\text{-}\textrm{Deciding}$}}
\medskip
\begin{proof}
By Lemma~\ref{lem-craig} any $\Sigma_n$-definable theory is equivalent with a $\Pi_{n-1}$-definable theory, and if that theory is $(n\!-\!1)$-consistent, then (extending $\q$) it cannot be $\Pi_n$-deciding by Corollary~\ref{cor-hajek2}.
\end{proof}

In fact, we can prove even a more general theorem here: no  $(n\!-\!1)$-consistent and $\Sigma_n$-definable extension of $\q$ can be  $\Pi_n$-deciding (because what was used in the above proof was the $(n\!-\!1)$-consistency of the theory); this is actually a generalization of  G\"odel-Rosser's incompleteness theorem.

\begin{corollary}\label{cor-gr}
No $\Sigma_n$-definable %($n\!\!>\!\!1$)
 extension of {\rm $\q$}  can be $\Pi_{n}$-deciding if it is $(n\!-\!1)$-consistent.
 \hfill\ding{111}
\end{corollary}
\centerline{
\fbox{$\textrm{{\bf Corollary~\ref{cor-gr}}}:\qquad  \q\subseteq T  \;\; \& \;\;  {\sf Axioms}_T\!\in\!\Sigma_{n} \;\; \& \;\; (n\!-\!1)\text{-}{\sf Con}(T) \;\;\; \Longrightarrow \;\;\; T\not\in\Pi_{n}\text{-}\textrm{Deciding}$}}

\section{Rosser's Theorem Optimized}
Rosser's Trick is one of the most fruitful tricks in Mathematical Logic and Recursion Theory (cf. \cite{smullyan}). One of its uses is getting rid of the condition of $\omega$-consistency (or 1-consistency or equivalently  consistency with the set of true $\Pi_1$-sentences) from the hypothesis of G\"odel's first incompleteness theorem. Thus, G\"odel-Rosser's incompleteness theorem (see e.g. \cite{hajekpudlak,smith,smullyan}) can be depicted as:

\medskip
\centerline{
\fbox{$\textrm{G\"odel--Rosser (1936)}:\qquad  \q\subseteq T  \;\; \& \;\;  {\sf Axioms}_T\!\in\!\Sigma_{1} \;\; \& \;\; {\sf Con}(T) \;\;\; \Longrightarrow \;\;\; T\not\in\Pi_{1}\text{-}\textrm{Deciding}$}}
\medskip

\noindent
In the light of our above mentioned results it is natural to expect a generalization of this theorem to higher levels (to $\Sigma_n$ or $\Pi_n$ definable theories); alas (by the following theorem for $n=0$) there can bo no such generalization for Rosser's  Theorem.

\begin{theorem}\label{thm-rosser}
There exists a complete (and consistent) and $\Sigma_{n+2}$-definable extension of {\rm $\q+\Pi_n\text{-}{\rm Th}(\mathbb{N})$}.
\end{theorem}
\begin{proof}
That there exists a complete $\Sigma_2$-definable extension of $\q$ is almost a classical fact; see \cite{smorynski}.
Here, we generalize this result to $\q+\Pi_n\text{-}{\rm Th}(\mathbb{N})$.
Let the theory $S$ be $\q$ when $n=0$ and be $\q+\Pi_n\text{-}{\rm Th}(\mathbb{N})$ when $n>0$ (note that $\Pi_0\text{-}{\rm Th}(\mathbb{N})\subseteq\q$).
Theory $S$ can be completed by Lindenbaum's Lemma as follows: for an enumeration of all the sentences $\varphi_0,\varphi_1,\varphi_2,\cdots$  take $T_0=S$, and let $T_{n+1}=T_n+\varphi_n$ if ${\sf Con}(T_n+\varphi_n)$ and let $T_{n+1}=T_n+\neg\varphi$ otherwise [if $\neg{\sf Con}(T_n+\varphi_n)$]. Then the
theory  $T^\ast=\bigcup_{n\in\mathbb{N}} T_n$  is a complete extension of $S$; below we show the $\Sigma_{n+2}$-definability of $T^\ast$.
An enumeration of all the sentences  can be defined by a $\Sigma_0$-formula such as the following expression for
``$x$ is the (G\"odel number of the) $u^{\rm th}$ sentence'':

$${\sf Sent}\text{-}{\sf List}(x,u)=\big[{\sf Sent}(u)\wedge x=u\big]\vee\big[\neg{\sf Sent}(u)\wedge x=\ulcorner 0=0\urcorner\big].$$

Now, ${\sf Axioms}_{T^\ast}(x)$ can be defined by the following formula:

\bigskip
\bigskip
\bigskip

 $\exists y\Big[{\sf Seq}(y) \wedge [y]_{\ell{\rm en}(y)-1}=x \wedge\big(\forall u\!<\!\ell{\rm en}(y)\big[{\sf Sent}([y]_u)\big]\big)\wedge $

\hspace{12ex} $\forall u\!<\!\ell{\rm en}(y)\forall z\!\leqslant\!y\Big(\big({\sf Sent}\text{-}{\sf List}(z,u)\wedge{\sf Con'}(S+\langle y\!\downharpoonright\!u\rangle+z)\longrightarrow [y]_u=z\big)\wedge$

\hspace{30.5ex}  $\big({\sf Sent}\text{-}{\sf List}(z,u)\wedge\neg{\sf Con'}(S+\langle y\!\downharpoonright\!u\rangle+z)\longrightarrow  [y]_u=\neg z\big) \Big)\Big]$,

\noindent which is  $\Sigma_{n+2}$  because  the following formula (where $\texttt{q}$ is the G\"odel code of the conjunction of all the [finitely many] axioms of $\q$ and $\bot=[0\neq 0]$)

\noindent
${\sf Con'}(S\!+\!\langle y\!\downharpoonright\!u\rangle\!+\!z)\equiv\begin{cases} \forall v,w \big[{\sf ConjSeq}(v,\langle y\!\downharpoonright\!u\rangle)\!\rightarrow\!\neg{\sf Proof}(w,\texttt{q}\!\wedge\!v\!\wedge\!z\!\rightarrow
\ulcorner\bot\urcorner)\big] & \textrm{ if }n\!=\!0 \\ \forall t,v,w \big[\Pi_n\text{-}{\sf True}(t)\!\wedge\!{\sf ConjSeq}(v,\langle y\!\downharpoonright\!u\rangle)\!\rightarrow\!\neg{\sf Proof}(w,\texttt{q}\!\wedge\!t\!\wedge\!v\!\wedge\!z\!\rightarrow
\ulcorner\bot\urcorner)\big] & \textrm{ if }n\!>\!0\end{cases}$

\noindent
is  $\Pi_{n+1}$  since $\Pi_n\text{-}{\sf True}\in\Pi_n$ (and ${\sf ConjSeq},{\sf Proof}\!\in\!\Pi_0$).
\end{proof}

\subsection{Comparing $\Sigma_n$-Soundness with $n$-Consistency}
The assumptions on the theory $T$ used in Corollaries~\ref{cor-1st} and~\ref{cor-g2},  other than $\q\subseteq T \& {\sf Axioms}_T\!\in\!\Sigma_{n}$, are either consistency with the set of all true $\Pi_n$ sentences (or equivalently, $\Sigma_n$-soundness) or $n$-consistency of $T$ (cf. also Corollaries~\ref{cor-2nd} and~\ref{cor-gr}).
So, it is desirable to compare the assumptions of $\Sigma_n$-soundness and $n$-consistency used in these results.

\begin{proposition}\label{prop-compare}
{\rm (1)} If a theory is $\Sigma_n$-sound, then it is $n$-consistent.

{\rm (2)} If a $\Sigma_{n-1}$-complete theory is $n$-consistent, then it is $\Sigma_n$-sound. \end{proposition}

\begin{proof}
(1) Assume $T\vdash\exists x\psi(x)$ for some $\Sigma_n$-sound theory $T$ and some formula $\psi\in\Pi_{n-1}$. By the $\Sigma_n$-soundness of $T$,  $\mathbb{N}\models\exists x\psi(x)$, and so  $\mathbb{N}\models\psi(m)$ for some $m\in\mathbb{N}$. Now, $\psi(\overline{m})\in\Pi_{n-1}\text{-}{\rm Th}(\mathbb{N})$, and again by the $\Sigma_n$-soundness of $T$ we have  $T\not\vdash\neg\psi(\overline{m})$.

\quad\; (2) Assume $T\vdash\exists x\psi(x)$ for some $\Sigma_{n-1}$-complete and $n$-consistent theory $T$ and some formula $\psi\in\Pi_{n-1}$. By $n$-consistency, there exists some $m\in\mathbb{N}$ such that $T\not\vdash\neg\psi(\overline{m})$. By $\Sigma_{n-1}$-completeness, $\mathbb{N}\not\models\neg\psi(\overline{m})$; and so $\mathbb{N}\models\psi(\overline{m})$, whence $\mathbb{N}\models\exists x\psi(x)$.
\end{proof}

\begin{remark}\label{rmk-compare}{\rm
In fact, for $n=0,1,2$  the notions of $\Sigma_n$-soundness and $n$-consistency are equivalent for $\Sigma_1$-complete theories (see Theorems~5,25,30  of~\cite{isaacson}); but for $n\!\geqslant\!3$, $n$-consistency does not imply $\Sigma_n$-soundness. Even, $\omega$-consistency does not imply $\Sigma_3$-soundness (see Theorem~19 of~\cite{isaacson} proved by Kreisel in 1955). Generally,   $\Sigma_n$-soundness does not imply $(n+1)$-consistency: Let $\gamma$ be the true $\Pi_{n+1}$-sentence constructed in Theorem~\ref{thm-1st} for the theory $\q+\Pi_n\text{-}{\rm Th}(\mathbb{N})$ and put $S=T+\neg\gamma$. Now, $S$ is $\Sigma_n$-sound and not $(n+1)$-consistent, since for $\neg\gamma=\exists x\delta(x)\in\Sigma_{n+1}$ we have $S\vdash\exists x\delta(x)$ and for any $k\in\mathbb{N}$ we have $S\vdash\neg\delta(\overline{k})$ since $S$ is $\Sigma_{n+1}$-complete by Remark~\ref{rem-1} and $\neg\delta(\overline{k})\in\Sigma_{n+1}\text{-}{\rm Th}(\mathbb{N})$ (because, if $\mathbb{N}\not\models\neg\delta(\overline{k})$ then $\mathbb{N}\models\delta(\overline{k})$ and so  $\mathbb{N}\models\neg\gamma$ contradiction!).
}\hfill\ding{71}\end{remark}

\noindent\;\; $\Sigma_0$-Sound$\;\,\Longleftarrow \Sigma_1$-Sound$\;\;\;\,\,\Longleftarrow \Sigma_2$-Sound$\;\;\;\;\,\Longleftarrow\Sigma_3$-Sound
$\;\;\;\,\,\Longleftarrow\cdots
\Sigma_n$-Sound $\cdots\,\,\,\,\,\Longleftarrow$ Sound

\noindent\;\; \qquad $\Updownarrow$ \qquad \qquad \qquad $\Updownarrow$
\qquad \qquad \qquad $\Updownarrow$ \qquad \quad \qquad \qquad $\Downarrow$
\qquad \qquad \qquad \qquad $\Downarrow$ \qquad \qquad  $\not\!\!\diagdown\!\!\!\!\!\!\!\nwarrow$
\quad $\Downarrow$

\noindent\;\; ${\sf Consistent}\Longleftarrow 1$-${\sf Consistent}\Longleftarrow 2$-${\sf Consistent}\Longleftarrow 3$-${\sf Consistent}\Longleftarrow\cdots n$-${\sf Consistent}\cdots\Longleftarrow\omega$-${\sf Consistent}$

\begin{corollary}\label{cor-rosser}
{\rm (1)} There exists a complete  extension of {\rm $\q$} which is $\Sigma_{n+2}$-definable and  consistent with $\Pi_{n}\text{-}{\rm Th}(\mathbb{N})$ (and so $n$-consistent).

\qquad\,\qquad\,\quad   {\rm (2)} There exists a complete  extension of {\rm $\q$} which is $\Pi_{n+1}$-definable and  consistent with $\Pi_{n}\text{-}{\rm Th}(\mathbb{N})$ (and so $n$-consistent).
\end{corollary}
\centerline{
\fbox{$\textrm{{\bf Corollary~\ref{cor-rosser}}(1)}:\;  \q\subseteq T  \; \& \;  {\sf Axioms}_T\!\in\!\Sigma_{n+2} \; \& \; \big[{\sf Con}\big(T+\Pi_{n}\text{-}{\rm Th}(\mathbb{N})\big)\bigvee n\text{-}{\sf Con}(T)\big] \;\; \not\!\Longrightarrow \;\; T\not\in\textrm{Complete}$}}
\centerline{
\fbox{$\textrm{{\bf Corollary~\ref{cor-rosser}}(2)}:\;  \q\subseteq T  \; \& \;  {\sf Axioms}_T\!\in\!\Pi_{n+1} \; \& \; \big[{\sf Con}\big(T+\Pi_{n}\text{-}{\rm Th}(\mathbb{N})\big)\bigvee n\text{-}{\sf Con}(T)\big] \;\; \not\!\Longrightarrow \;\; T\not\in\textrm{Complete}$}}
\medskip
\begin{proof}
 (1) The $\Sigma_{n+2}$-definable and complete extension of $\q$ in Theorem~\ref{thm-rosser} contains $\Pi_{n}\text{-}{\rm Th}(\mathbb{N})$, and so, being $\Sigma_n$-sound, is $n$-consistent by Proposition~\ref{prop-compare}.

\noindent  (2) The $\Sigma_{n+2}$-definable theory of part (1) is equivalent with a $\Pi_{n+1}$-definable theory by Lemma~\ref{lem-craig}.
\end{proof}

\subsection{A Note on the Constructiveness of the Proofs}
It is interesting to note that for $n\!\geqslant\!3$ all the incompleteness proofs (presented here) with the assumption of $\Sigma_n$-soundness are  constructive (i.e., the independent sentence can be effectively constructed from the given $\Sigma_n$-sound theory satisfying the conditions of $\Sigma_n/\Pi_n$ definability), while all the incompleteness proofs (here) with the assumption of $n$-consistency are all non-constructive (i.e., the independent sentence is not constructed explicitly, and only its mere existence is proved).
Our final result contains a bit of a surprise: even though the proof of Corollary~\ref{cor-gr} is not constructive, no one can present a constructive proof for it.

\begin{theorem}[Non-Constructivity of ${\bf n}$-Consistency Incompleteness]\label{thm-nrec} Let $n\!\geqslant\!3$ be fixed.
There is no recursive function $f$ \textup{(}even with the  oracle $\emptyset^{(n)}$\textup{)} such that when $m$ is a  \textup{(}G\"odel code of a\textup{)} $\Sigma_{n+1}$-formula which defines an   $n$-consistent extension of $\q$, then $f(m)$ is a \textup{(}G\"odel code of a\textup{)} $\Pi_{n+1}$-sentence independent from that theory.
\end{theorem}
\begin{proof}
Assume that there is   an $\emptyset^{(n)}$-recursive function $f$ such that for any given $\Sigma_{n+1}$-formula $\Psi(x)$ if the theory $\mathcal{T}_\Psi=\{\alpha\mid \mathbb{N}\models\Psi(\ulcorner\alpha\urcorner)\}$ is an $n$-consistent extension of $\q$ then $f(\ulcorner\Psi\urcorner)$ is (the G\"odel code of) a $\Pi_{n+1}$-sentence such that $\mathcal{T}_\Psi\not\vdash f(\ulcorner\Psi\urcorner)$ and $\mathcal{T}_\Psi\not\vdash\neg f(\ulcorner\Psi\urcorner)$. The $\omega$-consistency of $\q$ with $x$ can be written  by the   $\Pi_3$-formula   $\boldsymbol\omega\text{-}{\sf Con}_{\q}(x)=\forall\chi\big[\exists z{\sf Proof}\big(z,\texttt{q}\!\wedge\!x\!\rightarrow\!\exists v\chi(v)\big)\rightarrow\exists v \forall z \neg{\sf Proof}\big(z,{\tt q}\!\wedge\!x\!\rightarrow\!\neg\chi(\overline{v})\big)]$,  where ${\tt q}$ is the G\"odel code of the conjunction of the finitely many axioms of $\q$ (see the Proof of Theorem~\ref{thm-rosser}).
By $\emptyset^{(n)}$-recursiveness of $f$ the expressions $y\!=\!f(x)$ and $f(z)\!\downarrow$ can be written by $\Sigma_{n+1}$-formulas (see e.g.~\cite{robic}). By Diagonal Lemma there exists some $\Sigma_{n+1}$-formula $\Theta(x)$ such that

\begin{tabular}{r c l l}
$\Theta(x)$ & $\equiv$ & $\big[f(\ulcorner\Theta\urcorner)\!\downarrow \wedge \; \boldsymbol\omega\text{-}{\sf Con}_{\q}\big(f(\ulcorner\Theta\urcorner)\big) \wedge \,  \big(x\!=\!f(\ulcorner\Theta\urcorner)\vee x\!=\!{\tt q}\big)\big]$ & $\bigvee$ \\
 &  &  $\big[f(\ulcorner\Theta\urcorner)\!\downarrow \wedge \; \neg\boldsymbol\omega\text{-}{\sf Con}_{\q}\big(f(\ulcorner\Theta\urcorner)\big) \wedge \,  \big(x\!=\!\neg f(\ulcorner\Theta\urcorner)\vee x\!=\!{\tt q}\big)\big]$ & $\bigvee$ \\
 & & $(x={\tt q})$. &
\end{tabular}

\noindent
Now, if  $f(\ulcorner\Theta\urcorner)\!\uparrow$ then $\Theta(x)\equiv (x\!=\!{\tt q})$ and so $\mathcal{T}_\Theta\!=\!\q$  is an $n$-consistent extension of $\q$, whence  %we should have
$f(\ulcorner\Theta\urcorner)\!\downarrow$; contradiction. Thus, $f(\ulcorner\Theta\urcorner)\!\downarrow$. If $\q\cup\{f(\ulcorner\Theta\urcorner)\}$ is $\omega$-consistent then we have $\Theta(x)\equiv (x\!=\!f(\ulcorner\Theta\urcorner) \vee x\!=\!{\tt q})$ and so $\mathcal{T}_\Theta=\q\cup
\{f(\ulcorner\Theta\urcorner)\}$ is an $n$-consistent extension of $\q$, whence $f(\ulcorner\Theta\urcorner)$ should be independent from it; contradiction. So, $\q\cup\{f(\ulcorner\Theta\urcorner)\}$ is not $\omega$-consistent; then by \cite[Theorem~21]{isaacson} (which states that for any $\omega$-consistent theory $S$ and any sentence $X$ either $S\cup\{X\}$ or $S\cup\{\neg X\}$ is $\omega$-consistent) the theory $\q\cup\{\neg f(\ulcorner\Theta\urcorner)\}$ should be $\omega$-consistent.
But in this case we have
$\Theta(x)\equiv (x\!=\!\neg f(\ulcorner\Theta\urcorner)\vee x\!=\!{\tt q})$ and so $\mathcal{T}_\Theta=\q\cup\{\neg f(\ulcorner\Theta\urcorner)\}$ is an $n$-consistent extension of $\q$, whence $f(\ulcorner\Theta\urcorner)$ should be independent from it; contradiction again. Thus there can be no such $\emptyset^{(n)}$-recursive function.
\end{proof}

\begin{remark}[Optimality of Theorem~\ref{thm-nrec}]{\rm
Even though, by Theorem~\ref{thm-nrec}, there does not exist any $\emptyset^{(n)}$-recursive function (for $n\!>\!2$) which can output an independent $\Pi_{n+1}$-sentence for a given $\Sigma_{n+1}$-definable and $n$-consistent extension of $\q$, there indeed exists some $\emptyset^{(n+1)}$-recursive function which can find such an independent $\Pi_{n+1}$-sentence (for a given $\Sigma_{n+1}$-definition of an  $n$-consistent extension of $\q$): By having an access to the oracle $\emptyset^{(n+1)}$ for  a given ${\sf Ax}_T\!\in\!\Sigma_{n+1}$, provability (or unprovability) in $T$ of a given sentence is decidable; thus (since by Corollary~\ref{cor-gr}  there must exist some $\Pi_{n+1}$-sentence independent from the theory $T$) by an exhaustive  search through all the $\Pi_{n+1}$-sentences such an independent sentence can be eventually found.
}\hfill\ding{71}\end{remark}

%%%%%%%%%%%%%%%%%%%%%%%%%%%%%%

\section{Conclusions}\label{sec-concl}

Summing up, G\"odel first incompleteness theorem in its semantic form, which states the $\Pi_1$-incompleteness of any sound and $\Sigma_1$-definable extension of $\q$, can be generalized to show that any sound and $\Sigma_n$-definable extension of $\q$ is $\Pi_n$-incomplete. Also, G\"odel's original first incompleteness theorem, which is equivalent to the $\Pi_1$-undecidability  of  any $\Sigma_1$-sound and  $\Sigma_1$-definable extension of $\q$, can be generalized to show that no $\Sigma_n$-sound and $\Sigma_n$-definable extension of $\q$  is $\Pi_n$-deciding (here actually $\Sigma_{n-1}$-soundness suffices by Rosser's Trick). Finally, Rosser's incompleteness theorem, which states the $\Pi_1$-undecidability  of any consistent and $\Sigma_1$-definable extension of $\q$, cannot be generalized to definable theories, not even to $\Pi_1$-definable ones.  Concluding,  we have the following table for $n\!>\!1$ which shows our results in a viewable perspective:

\begin{table}[h]
\begin{center}
    \begin{tabular}{|l||lll|}
  \hline
 G\"odel's 1$^{\rm st}$  {\footnotesize (Semantic)} & $\q\subseteq T$ \& ${\sf Axioms}_T\!\in\!\Sigma_1$ \& %\quad
 $T$ is ({\small $\Sigma_\infty\!$})Sound &\!$\Longrightarrow$\!& $T$\, $\not\in$\, $\Pi_1-$Complete    \\
   \hline
 Theorem~\ref{prop-godel}  & $\q\subseteq T$ \& ${\sf Axioms}_T\!\in\!\Sigma_{n}$ \&  %\quad
 $T$ is ({\small $\Sigma_\infty\!$})Sound &\!$\Longrightarrow$\!& $T$\, $\not\in$\, $\Pi_{n}-$Complete    \\
   \hline
   G\"odel's 1$^{\rm st}$   & $\q\subseteq T$ \& ${\sf Axioms}_T\!\in\!\Sigma_{1}$ \&  $T$ is \,$\Sigma_1-$Sound &\!$\Longrightarrow$\!& $T$\, $\not\in$\, $\Pi_{1}-$Deciding     \\
    \hline
   Corollary~\ref{cor-1st}  & $\q\subseteq T$ \& ${\sf Axioms}_T\!\in\!\Sigma_{n}$ \& $T$ is \,$\Sigma_n\!-$Sound  &\!$\Longrightarrow$\!& $T$\, $\not\in$\, $\Pi_{n}-$Deciding     \\
  \hline
         G\"odel--Rosser & $\q\subseteq T$ \& ${\sf Axioms}_T\!\in\!\Sigma_{1}$ \& $T$ is \,$\Sigma_0-$Sound &\!$\Longrightarrow$\!& $T$\, $\not\in$\, $\Pi_{1}-$Deciding     \\
  \hline
   Corollary~\ref{cor-2nd}  & $\q\subseteq T$ \& ${\sf Axioms}_T\!\in\!\Sigma_{n}$ \& $T$ is  $\Sigma_{n\!-\!1}\!$ Sound  &\!$\Longrightarrow$\!& $T$\, $\not\in$\, $\Pi_{n}-$Deciding     \\
   \hline
     Corollary~\ref{cor-rosser}(1)  & $\q\subseteq T$ \& ${\sf Axioms}_T\!\in\!\Sigma_{n}$ \& $T$ is  $\Sigma_{n\!-\!2}\!$ Sound  & $\not\!\Longrightarrow$\!& $T$\, $\not\in$\, Complete     \\
   \hline
  \hline
   G\"odel's 1$^{\rm st}$   & $\q\subseteq T$ \& ${\sf Axioms}_T\!\in\!\Sigma_{1}$ \& \;\qquad\;  $1\text{-}{\sf Con}(T)$ &\!$\Longrightarrow$\!& $T$\, $\not\in$\, $\Pi_{1}-$Deciding     \\
    \hline
    Corollary~\ref{cor-g2}   & $\q\subseteq T$ \& ${\sf Axioms}_T\!\in\!\Sigma_{n}$ \& \,\qquad\;   $n\text{-}{\sf Con}(T)$ &\!$\Longrightarrow$\!& $T$\, $\not\in$\, $\Pi_{n}-$Deciding     \\
    \hline
    G\"odel--Rosser & $\q\subseteq T$ \& ${\sf Axioms}_T\!\in\!\Sigma_{1}$ \& \qquad\;\quad\,  ${\sf Con}(T)$ &\!$\Longrightarrow$\!& $T$\, $\not\in$\, $\Pi_{1}-$Deciding     \\
  \hline
  Corollary~\ref{cor-gr}   & $\q\subseteq T$ \& ${\sf Axioms}_T\!\in\!\Sigma_{n}$ \& \; $(n\!-\!1)\text{-}{\sf Con}(T)$ &\!$\Longrightarrow$\!& $T$\, $\not\in$\, $\Pi_{n}-$Deciding     \\
    \hline
            Corollary~\ref{cor-rosser}(1)  & $\q\subseteq T$ \& ${\sf Axioms}_T\!\in\!\Sigma_{n}$ \& \; $(n\!-\!2)\text{-}{\sf Con}(T)$  & $\not\!\Longrightarrow$\!& $T$\, $\not\in$\, Complete     \\
   \hline
   \end{tabular}
\vspace{-1.5em}
\end{center}
 \end{table}

\noindent To complete the picture here are  the $\Pi$ version of the results for $m\!>\!0$:

   \begin{table}[h]
\begin{center}
    \begin{tabular}{|l||lll|}
   \hline
  Theorem~\ref{thm-1st}  & $\q\subseteq T$ \& ${\sf Axioms}_T\!\in\!\Pi_{m}$ \& $T$ is \,$\Sigma_m\!-$Sound  &\!$\Longrightarrow$\!& $T$\, $\not\in$\, $\Pi_{m+1}-$Deciding     \\
          \hline
   Corollary~\ref{cor-rosser}(2)  & $\q\subseteq T$ \& ${\sf Axioms}_T\!\in\!\Pi_{m}$ \& $T$ is $\Sigma_{m\!-\!1}\!$ Sound   & $\not\!\Longrightarrow$\!& $T$\, $\not\in$\, Complete    \\
   \hline
     Corollary~\ref{cor-hajek2}  & $\q\subseteq T$ \& ${\sf Axioms}_T\!\in\!\Pi_{m}$ \& \,\qquad\; $m\text{-}{\sf Con}(T)$  &\!$\Longrightarrow$\!& $T$\, $\not\in$\, $\Pi_{m+1}-$Deciding     \\
   \hline
   Corollary~\ref{cor-rosser}(2)  & $\q\subseteq T$ \& ${\sf Axioms}_T\!\in\!\Pi_{m}$ \& \; $(m\!-\!1)\text{-}{\sf Con}(T)$   & $\not\!\Longrightarrow$\!& $T$\, $\not\in$\, Complete    \\
   \hline
\end{tabular}
\vspace{-1.5em}
\end{center}
  \end{table}

%%%%%%%%%%%%%%%%%%%%%%%%%

%\newpage

\end{document}